\documentclass[a4paper,11pt]{article}
\usepackage{stmaryrd}
\usepackage{bbm}
\usepackage[top=2.5cm, bottom=2.6cm, left=3cm, right=2.6cm]{geometry}
\usepackage{amsfonts}
\usepackage{mathrsfs}
\usepackage{amsmath}
\usepackage{latexsym}
\usepackage{color}
\usepackage{amscd}
\usepackage{amssymb}
\usepackage{amsthm}
\usepackage{latexsym}
\usepackage{indentfirst}

\newtheorem{theorem}{Theorem}[section]

\newtheorem{lemma}{Lemma}[section]

\begin{document}

\title{Affine cellularity of affine Brauer algebras}
\author{Weideng Cui }
\date{}
\maketitle \abstract{We show that the affine Brauer algebras are affine cellular algebras in the sense of Koenig and Xi.}\\

 \large

\medskip
\section{Introduction}

In order to approach the fundamental problem of classifying the irreducible representations of a given finite-dimensional algebra, the concept of cellularity, defined by Graham and Lehrer in [GL], has proven extremely useful. Examples of cellular algebras include many finite-dimensional Hecke algebras.

Recently, Koenig and Xi in [KX3] has generalized this concept to algebras over a noetherian domain $k$ of not necessarily finite dimension, by introducing  the notion of an affine cellular algebra. The most important class of examples of affine cellular algebras, which has been discussed in [KX2], is given by the extended affine Hecke algebras of type $A.$ Recently, Guilhot and Miemietz have proved that affine Hecke algebras of rank two with generic parameters are affine cellular in [GuM]; Kleshchev and Loubert have proved that KLR algebras of finite type are affine cellular in [KL]. In [C1] and [C2], we have shown that the BLN-algebras which have been introduced by McGerty in [Mc], and the affine $q$-Schur algebras which have been defined by Lusztig in [L], are affine cellular algebras.

The Brauer algebra is introduced as an enlargement of the symmetric group algebra and is in Schur-Weyl duality with the orthogonal or symplectic group. The BMW algebra, which is introduced in [BW] and [Mu], can be considered as a deformation of the Brauer algebra by replacing the symmetric groups by their Hecke algebras. In [MW], they have constructed a basis of the BMW algebra, which is indexed by Brauer $n$-diagrams. In [Xi], using this basis, he has shown that the basis is in fact a cellular basis and thus the BMW algebra is cellular (see also [E]). 

The cyclotomic Brauer algebra is a corresponding enlargement of the complex reflection group algebra of type $G(m, 1, n).$ This has been introduced by H\"{a}ring-Oldenburg in [HO] as a specialization of the cyclotomic BMW algebra, and has been studied by various authors (see for example [RY, RX, BCV]). In [RY], they have given a cellular basis of the cyclotomic Brauer algebra and thus it is cellular (see also [BCV]).

The affine Brauer algebra has been introduced by H\"{a}ring-Oldenburg in [HO] as a specialization of the affine BMW algebra, and has been studied in [GH]. In [GH], they have introduced the colored Brauer $n$-diagrams, using which they have defined the affine Brauer algebra. In the subsequent papers [GM] and [G], they have constructed several bases of the affine BMW algebras which are all indexed by the colored Brauer $n$-diagrams.

In this note, We will show that the affine Brauer algebras are affine cellular algebras in the sense of Koenig and Xi. Our proof relies on the fact that the group algebra of the extended affine Weyl group of type $A$ is an affine cellular algebra over $\mathbb{Z}$.

The organization of this note is as follows. In Section 2, we introduce affine cellular algebras. In Section 3, we will recall the definitions of Brauer algebras and affine Brauer algebras. In Section 4, we prove our main result Theorem 4.1.

\section{Affine cellular algebras}

Let $k$ be a noetherian domain. For a $k$-algebra $A$, a $k$-linear anti-automorphism $i$ of $A$ satisfying $i^{2}=id_{A}$ will be called a $k$-involution on $A.$ For two $k$-modules $V$ and $W,$ we denote by $\tau$ the map $V\otimes W\rightarrow W\otimes V$ given by $\tau(v\otimes w)=w\otimes v.$ If $B=k[x_{1},\ldots,x_{t}]/I$ for some ideal $I$ in a polynomial ring in finitely many variables $x_{1},\ldots, x_{t}$ over $k$, then $B$ is called an affine $k$-algebra.

\hspace*{-0.5cm}$\mathbf{Definition~2.1.}$ (see [KX3, Definition 2.1]) Let $A$ be a unitary $k$-algebra with a $k$-involution $i$. A two-sided ideal $J$ in $A$ is called an affine cell ideal if and only if the following data are given and the following conditions are satisfied:\vskip3mm
$(1)$ We have $i(J)=J.$

$(2)$ There exist a free $k$-module of finite rank and an affine $k$-algebra $B$ with a $k$-involution $\sigma$ such that $\Delta :=V\otimes_{k} B$ is an $A$-$B$-bimodule, where the right $B$-module structure is induced by the right regular $B$-module $B_{B}$.

$(3)$ There is an $A$-$A$-bimodule isomorphism $\alpha :J\rightarrow \Delta\otimes_{B}\Delta',$ where $\Delta'=B\otimes_{k}V$ is a $B$-$A$-bimodule with the left $B$-module induced by the left regular $B$-module ${}_{B}B$ and with the right $A$-module structure defined by $(b\otimes v)a :=\tau(i(a)(v\otimes b))$ for $a\in A$, $b\in B$ and $v\in V$, such that the following diagram is commutative:

\[\begin{CD}
j   @>\alpha>>\Delta\otimes_{B}\Delta'\\
@ViVV                  @VVv\otimes b\otimes_{B}b'\otimes w\mapsto w\otimes \sigma(b')\otimes_{B}\sigma(b)\otimes vV\\
J         @>\alpha>>   \Delta\otimes_{B}\Delta'
\end{CD}\].

The algebra $A$ together with its $k$-involution $i$ is called affine cellular if and only if there is a $k$-module decomposition $A=J_{1}'\oplus J_{2}'\oplus\cdots J_{n}'$ (for some $n$) with $i(J_{l}')=J_{l}'$ for $1\leq l\leq n,$ such that, setting $J_{m}: =\bigoplus_{l=1}^{m}J_{l}',$ we have a chain of two-sided ideals of $A$: $0=J_{0}\subset J_{1}\subset J_{2}\subset\cdots\subset J_{n}=A$, where each $J_{m}'=J_{m}/J_{m-1}$ ($1\leq m\leq n$) is an affine cell ideal of $A/J_{m-1}$ (with respect to the involution induced by $i$ on the quotient).\vskip3mm

To prove that an algebra is affine cellular, one canonical way is to use the definitions. The following lemma, which can be regarded as an affine version of the iterated inflation developed in [KX1] and [KX2], provides another possibility to verify the affine cellularity of a given algebra, and describes also the structures of a general affine cellular algebra.

\begin{lemma}
Let $k$ be a noetherian domain, $A$ a unitary $k$-algebra with a $k$-involution $i.$ Suppose that there is a decomposition $$A=\bigoplus_{j=1}^{m} V_{j}\otimes_{k} V_{j}\otimes_{k}B_{j}~~~(\mathrm{direct~ sums ~of}~k\mathrm{-modules})$$
where $V_{j}$ is a free $k$-module of finite rank and $B_{j}$ is an affine cellular algebra with respect to an involution $\sigma_{j}$ and a cell chain $J_{1}^{(j)}\subset J_{2}^{(j)}\subset \cdots J_{s_{j}}^{(j)}=B_{j}$ for each $j.$ Define $J_{t}=\bigoplus_{j=1}^{t} V_{j}\otimes_{k} V_{j}\otimes_{k}B_{j}.$ Assume that the restriction of $i$ on $V_{j}\otimes_{k} V_{j}\otimes_{k}B_{j}$ is given by $w\otimes v\otimes b \mapsto v\otimes w \otimes \sigma_{j}(b).$ If for each $j$ there is a bilinear form $\phi_{j} : V_{j}\otimes_{k} V_{j}\rightarrow B_{j}$ such that $\sigma_{j}(\phi_{j}(w,v))=\phi_{j}(v,w)$ for all $w, v\in V_{j}$ and that the multiplication of two elements in $V_{j}\otimes_{k} V_{j}\otimes_{k}B_{j}$ is governed by $\phi_{j}$ modulo $J_{j-1},$ that is, for $x, y, u, v\in V_{j}$ and $b, c\in B_{j},$ we have $(x\otimes y\otimes b)(u\otimes v\otimes c)=x\otimes v\otimes b\phi_{j}(y,u)c$ modulo the ideal $J_{j-1},$ and if $V_{j}\otimes V_{j}\otimes J_{l}^{(j)}+J_{j-1}$ is an ideal in $A$ for all $l$ and $j,$ then $A$ is an affine cellular algebra.
\end{lemma}
\begin{proof}
Since $J_{1}^{(j)}\subset J_{2}^{(j)}\subset \cdots J_{s_{j}}^{(j)}=B_{j},~ j=1,2,\ldots,m$ is a cell chain for the given affine cellular algebra $B_{j},$ we can check that the following chain of ideals in $A$ satisfies all conditions in Definition 2.1: \vskip2mm $\hspace*{0.4cm}V_{1}\otimes V_{1}\otimes J_{1}^{(1)}\subset \cdots \subset V_{1}\otimes V_{1}\otimes J_{s_{1}}^{(1)}\subset V_{1}\otimes V_{1}\otimes B_{1}\oplus V_{2}\otimes V_{2}\otimes J_{1}^{(2)}$

$\subset V_{1}\otimes V_{1}\otimes B_{1}\oplus V_{2}\otimes V_{2}\otimes J_{2}^{(2)}\subset \cdots \subset V_{1}\otimes V_{1}\otimes B_{1}\oplus V_{2}\otimes V_{2}\otimes B_{2}$

$\subset \cdots \subset \bigoplus_{j=1}^{m-1} V_{j}\otimes V_{j}\otimes B_{j}\oplus V_{m}\otimes V_{m}\otimes J_{1}^{(m)}\subset \cdots$

$\subset \bigoplus_{j=1}^{m-1} V_{j}\otimes V_{j}\otimes B_{j}\oplus V_{m}\otimes V_{m}\otimes J_{s_{m}}^{(m)}=A.$\vskip2mm

In fact, we take a fixed non-zero element $v_{j}\in V_{j}$ and suppose that $\alpha : J_{t}^{(j)}\rightarrow (V_{t}^{(j)}\otimes B_{t}^{(j)})\otimes_{B_{t}^{(j)}} (B_{t}^{(j)}\otimes V_{t}^{(j)})$ is the $B_{j}$-bimodule isomorphism in the definition of the affine cell ideal $J_{t}^{(j)}$ for the affine cellular algebra $B_{j}.$ Define $$\beta : V_{j}\otimes V_{j}\otimes J_{t}^{(j)}\rightarrow (V_{j}\otimes v_{j}\otimes V_{t}^{(j)}\otimes B_{t}^{(j)}) \otimes_{B_{t}^{(j)}} (B_{t}^{(j)}\otimes V_{j}\otimes v_{j}\otimes V_{t}^{(j)})$$$$u\otimes v\otimes x\rightarrow \sum_{l}(u\otimes v_{j}\otimes x_{l}^{(1)}\otimes b_{l}^{(1)})\otimes (b_{l}^{(2)}\otimes v\otimes v_{j}\otimes x_{l}^{(2)}),$$
where $u, v\in V_{j},$ $x\in J_{t}^{(j)}$ and $\alpha(x)=\sum_{l}(x_{l}^{(1)}\otimes b_{l}^{(1)})\otimes (b_{l}^{(2)}\otimes x_{l}^{(2)}).$ Then one can verify that $\beta$ is an $A$-bimodule isomorphism and makes the corresponding diagram in the definition of affine cell ideals commutative. Here $V_{j}\otimes V_{j}\otimes J_{t}^{(j)}$ is an affine cell ideal in the corresponding quotient of $A$. Thus $A$ is an affine cellular algebra.
\end{proof}

\section{Brauer algebras and Affine Brauer algebras}

\subsection{Brauer algebras}
In this subsection, we recall the definition of Brauer algebras and introduce some notations for the later use.

Let $n\in \mathbb{N},$ and let $\mathbb{Z}[\delta]$ be the polynomial ring in one variable $\delta$ over the integers.

The Brauer algebra $B_{\mathbb{Z}[\delta]}(n, \delta)$, or written briefly $B(n, \delta)$, has as $\mathbb{Z}[\delta]$-linear basis the set of all partitions of the set $S=\{1,2,\ldots,n,1',2',\ldots,n'\}$ into two-element subsets (here the cardinality of $S$ is $2n$). As usual we shall represent the basis element by a diagram in a rectangle of the plane, where the top row has n vertices marked by $1, 2, \ldots, n$ from left to right; and the bottom row is numbered by $1', 2',\ldots, n'$ from left to right. If $i$ and $j$ are in the same subset, we draw a line between $i$ and $j.$ We call the corresponding diagram a Brauer $n$-diagram and denote by $B_{n}$ the set of all Brauer $n$-diagrams.

The multiplication in the Brauer algebra $B(n, \delta)$ is just the concatenation of two diagrams with a coefficient counting the number of cycles produced by forming the concatenation. Let $D_{1}$ and $D_{2}$ be two Brauer $n$-diagrams, we can compose $D_{1}$ and $D_{2}$ to get another Brauer $n$-diagram $D_1 \circ D_2$ by replacing $D_1$ and $D_2$ and joining the corresponding points; interior loops are deleted. Let $r$ denote the number of deleted loops, then we define $D_1\cdot D_2=\delta^{r} D_1\circ D_2.$

We define an order on $S$ by $1< 2<\cdots < n, n'<(n-1)'<\cdots< 2'<1',$ and $i< j'$ for all $1\leq i,j\leq n.$ For $d\in B_{n},$ we write $t(d)$ for the number of `through strings', that is, vertical lines which connect a top vertex with a bottom vertex.

Generally, given an arbitrary ring $R$ with identity and an element $\delta\in R,$ we can define the Brauer algebra $B_{R}(n, \delta)$ over $R$ by using the Brauer $n$-diagrams as $R$-basis. The following result is well-known in [GL]; see also [KX2].

\begin{theorem}The Brauer algebra $B_{R}(n, \delta)$ is cellular for any ring $R$ and $\delta\in R.$\end{theorem}

\subsection{Affine Brauer algebras}
In this subsection, we will recall the definition of the affine Brauer algebras given in [GH], which can be considered as a sort of `wreath product' of $\mathbb{Z}$ with the Brauer algebras.

We define a colored Brauer $n$-diagram, which is a Brauer diagram in which each strand is endowed with an orientation and labeled by an element of the group $\mathbb{Z}.$ We denote by $\widehat{B_{n}}$ the set of all colored Brauer $n$-diagrams. Two labelings are regarded as the same if the orientation of a strand is reversed and the group element associated to the strand is inverted. Note that we only consider these diagrams without closed loops inside the frame.

Given $n\in \mathbb{N},$ let $R=\mathbb{Z}[\delta_{0}, \delta_{1},\ldots]$ be the ring with infinite variables $\delta_{0}, \delta_{1}, \ldots$ over the integers. We now define the affine Brauer algebra $\widehat{B}_{n, R}$ over $R.$ As a vector space, $\widehat{B}_{n, R}$ is the $R$-span of all colored Brauer $n$-diagrams. We define the product $x\cdot y$ of two colored Brauer $n$-diagrams $x$ and $y$ using the concatenation of $x$ above $y,$ where we identify the bottom nodes of $x$ and the top nodes of $y.$ More precisely, we first choose compatible orientations of the strands of $x$ and $y.$ Then we concatenate the diagrams and add the labels on each strand of the new diagram. Any closed loop in this new diagram can be oriented such that as the strand passes through the leftmost central node in the loop it points downwards. If this oriented loop is labeled by $\pm i$ for some $i\in \mathbb{N},$ then the diagram is set to be equal to $\delta_{i}$ times the same diagram with the loop moved.

Let $z$ be the colored Brauer $n$-diagram obtained by removing all the closed loops, and let $m_{i}$ be the number of deleted loops with labels $\pm i$ for each $i\in \{0, 1, 2,\ldots\},$ then we define $x\cdot y= (\delta_{0}^{m_{0}}\delta_{1}^{m_{1}}\delta_{2}^{m_{2}}\cdots) c.$

The affine Brauer algebra $\widehat{B}_{n, R}$ has an explicit presentation as follows. It can also be defined as the $R$-algebra generated by $\{s_i, e_i, t_j|~1\leq i\leq n-1,~ \mathrm{and}~ 1\leq j\leq n\}$ subject to the following relations:\vskip2mm

\hspace*{-0.5cm}(a) $s_{i}^{2}=1,$ for $1\leq i\leq n-1.$~~~~~~~~~~~~~~~~~~(b) $s_{i}s_{j}=s_{j}s_{i},$ if $|i-j|> 1.$

\hspace*{-0.5cm}(c) $s_{i}s_{i+1}s_{i}=s_{i+1}s_{i}s_{i+1},$ for $1\leq i< n-1.$

\hspace*{-0.5cm}(d) $s_{i}t_{j}=t_{j}s_{i},$ if $j\neq i, i+1.$~~~~~~~~~~~~~~~~~~~(e) $s_{i}e_{j}=e_{j}s_{i},$ if $|i-j|> 1.$

\hspace*{-0.5cm}(f) $e_{i}e_{j}=e_{j}e_{i},$ if $|i-j|> 1.$~~~~~~~~~~~~~~~~~~~(g) $e_{i}t_{j}=t_{j}e_{i},$ if $j\neq i, i+1.$

\hspace*{-0.5cm}(h) $t_{i}t_{j}=t_{j}t_{i},$ for $1\leq i,j\leq n.$~~~~~~~~~~~~~~~~(i) $s_{i}t_{i}=t_{i+1}s_{i},$ for $1\leq i< n.$

\hspace*{-0.5cm}(j) $e_{i}s_{i}=e_{i}=s_{i}e_{i},$ for $1\leq i\leq n-1.$

\hspace*{-0.5cm}(k) $s_{i}e_{i+1}e_{i}=s_{i+1}e_{i},$ for $1\leq i\leq n-2.$

\hspace*{-0.5cm}(l) $e_{i+1}e_{i}s_{i+1}=e_{i+1}s_{i},$ for $1\leq i\leq n-2.$~~~~(m)$e_{i}e_{j}e_{i}=e_{i},$ if $|i-j|= 1.$

\hspace*{-0.5cm}(n) $e_{i}t_{i}t_{i+1}=e_{i}=t_{i}t_{i+1}e_{i},$ for $1\leq i\leq n-1.$

\hspace*{-0.5cm}(o) $e_{i}t_{i}^{a}e_{i}=\delta_{a}e_{i},$ for $a\geq 0$ and $1\leq i\leq n-1.$\vskip2mm

One can prove that the two definitions of $\widehat{B}_{n, R}$ are equivalent by the arguments similar to those for BMW algebras in [MW].

\section{Affine cellularity of affine Brauer algebras}
In this section, we will show that $\widehat{B}_{n, R}$ is an affine cellular algebra.\vskip2mm

\hspace*{-0.5cm}$\mathbf{Definition~4.1.}$ Suppose that $n, k\in \mathbb{N}.$ An flat $(n, k)$-dangle is a partition of $\{1,2,\ldots,n\}$ into $k$ one-element subsets, and $(n-k)/2$ two-element subsets, here $k$ must be a natural number in $\{n,n-2,n-4,\ldots\}.$ An $(n,k)$-dangle is a flat $(n,k)$-dangle to which an integer $r\in \mathbb{Z}$ has been assigned to every subset of size 2.

We can represent a flat $(n,k)$-dangle $d$ by a set of $n$ nodes labeled by the set $\{1,2,\ldots,n\},$ where there is an arc joining $i$ to $j$ if $\{i, j\}\in d$, and there is a vertical line starting from $i$ if $\{i\}\in d.$ An $(n, k)$-dangle can be represented graphically by first labeling each arc of the underlying $(n, k)$-dangle and then giving it the following orientation: we let all one-element sets have a downward orientation and all two-element sets have a right orientation.

We denote by $D(n,k)$ the set of all $(n,k)$-dangles, and let $V(n,k)$ denote the vector space spanned by all $(n, k)$-dangles. We also define $V'(n, k)$ be a copy of $V(n, k),$ but draw the pictures dangling upward rather than down, and label the vertices by $1', 2',\ldots, n'.$ For each element $d\in D(n, k)\subset V(n, k),$ let $d'\in V'(n, k)$ be the corresponding element. For each $k\in \mathbb{N},$ we denote by $A_{k}$ the group algebra of the wreath product of $\mathbb{Z}$ and the symmetric group $\Sigma_{k}.$

\begin{lemma}
For each element $d\in \widehat{B_{n}},$ we can write $d$ uniquely as an element in $V(n,k)\otimes_{R} V'(n,k)\otimes_{R} A_{k}$ for some $k\in \mathbb{N}.$
\end{lemma}
\begin{proof}
Suppose that there are $k$ vertical lines in $d\in \widehat{B_{n}}.$ Then we have two $(n, k)$-dangles $d_1, d_2'$ (by cutting off all vertical lines) and an element $\widetilde{d}\in A_{k}.$ These data $(d_1, d_2',\widetilde{d})$ are uniquely determined by $d,$ where $\widetilde{d}$ is obtained in the following way: we can easily get a permutation $\pi(d)\in \Sigma_{k},$ if we also consider the labels on the k vertical lines then we get an element of $A_{k}$, which is just $\widetilde{d}$. Thus we can write $d$ as $d_1\otimes d_2'\otimes \widetilde{d}$. Clearly, if we are given $d_1,d_2'\in D(n,k),$ and $\widetilde{d}\in A_{k},$ we have a unique element $d\in \widehat{B_{n}}$ with the data $(d_1, d_2', \widetilde{d}).$
\end{proof}

Now we want to define an $R$-linear form $\varphi_{k}$ from $V'(n,k)\otimes_{R} V(n, k)$ to $A_{k}$. Given two basis elements $d_1'\in V'(n,k), d_2\in V(n, k).$ Then by Lemma 4.1, we can form the element $X_{j} :=d_j\otimes d_j'\otimes 1$ in $V(n, k)\otimes_{R} V'(n,k)\otimes_{R} A_{k}$ for $j=1,2$. Suppose that the product $X_1\cdot X_2$ is expressed as an $R$-linear combination of the basis elements in $\{d|~ d\in \widehat{B_{n}}\},$ say $\sum_{j}f_{j}' c_{j},$ where $f_{j}'\in R.$ It is clear from the multiplication in $\widehat{B}_{n, R}$ that this expression of $X_1\cdot X_2$ can be written as $\sum f_{i}c_i+a$ with $t(c_{i})=k$ and $a$ in the $R$-spanning of those basis elements $c$ with $t(c)< k.$ Now we rewrite $c_{i}$ as the form $c_{i,1}\otimes c_{i,2}'\otimes \widetilde{c_{i}}.$ Moreover, for those $c_{i}$ with $t(c_{i})=k,$ we know from the concatenation of two tangles that they are of the form $d_1\otimes d_2'\otimes \widetilde{c_{i}}.$ Here $X_1\cdot X_2$ can be written as $\sum d_1\otimes d_2'\otimes f_{i} \widetilde{c_{i}}+a.$ Now we can define $\varphi_{k}(d_1', d_2)$: $=\sum_{i} f_{i}\widetilde{c_{i}}\in A_{k}.$

Now we define $J_{k}$ to be the $R$-module generated by basis elements $d$ with $t(d)\leq k.$ It is clear that $J_{k}\subset J_{k+1}$ and $J_{k}$ is an ideal in $\widehat{B}_{n, R}$ for all $k.$
\begin{lemma}
If $c=c_1\otimes c_2'\otimes \widetilde{c}$ and $d=d_1\otimes d_2'\otimes \widetilde{d}$ with $c_1, d_1\in V(n, k),$ $c_2', d_2'\in V'(n, k)$ and $\widetilde{c}, \widetilde{d}\in A_{k},$ then $c\cdot d=c_1\otimes d_2'\otimes \widetilde{c}\varphi_{k}(c_2', d_1)\widetilde{d}$ $($$\mathrm{mod}$ $J_{k-2}$$).$
\end{lemma}
\begin{proof}
We may consider $c_2'$ as a $(k, n)$-dangle, and we now consider the concatenation of the two tangles $c_2'$ and $d_1,$ this gives us a $(k, k)$-dangle $T$. If we write this tangle $T$ as a linear combination of the basis elements $\{x|~ x\in \widehat{B_{k}}\},$ then we see that $T=I_{k}\otimes I_{k}\otimes \varphi_{k}(c_2', d_1)+a'$, where $\varphi_{k}(c_2', d_1)$ is the bilinear form defined as above, $a'\in J_{k-2}$ and $I_{k}=id_{k}$ is the $(k, k)$-dangle with $k$ vertical strands. So the product $c\cdot d$ is formed by a series of concatenations: $c_1\cdot \widetilde{c}\cdot c_2'\cdot d_1\cdot \widetilde{d} \cdot d_2'.$ Thus we have $c\cdot d=c_1\cdot \widetilde{c}\cdot \varphi_{k}(c_2', d_1)\cdot \widetilde{d} \cdot d_2'+a$ with $a\in J_{k-2}.$ This implies that $$c\cdot d\equiv c_1\otimes d_2'\otimes \widetilde{c}\varphi_{k}(c_2', d_1)\widetilde{d} ~~~~(\mathrm{mod}~J_{k-2}).$$ This proved the lemma.
\end{proof}

By Lemma 4.1, we have an $R$-module decomposition: $\widehat{B}_{n, R}=V(n,n)\otimes V'(n,n)\otimes A_{n}\oplus V(n,n-2)\otimes V'(n,n-2)\otimes A_{n-2}\oplus V(n,n-4)\otimes V'(n,n-4)\otimes A_{n-4}\oplus \cdots \oplus V(n,\varepsilon)\otimes V'(n,\varepsilon)\otimes A_{\varepsilon},$ where $\varepsilon$ is zero if $n$ is even, and $1$ if $n$ is odd. The following lemma tells us how to get an ideal in $\widehat{B}_{n, R}$ from an ideal in the group algebras.
\begin{lemma}
Let $I$ be an ideal in $A_{k}.$ Then $J_{k-2}+V(n, k)\otimes V'(n,k)\otimes I$ is an ideal in $\widehat{B}_{n, R}.$
\end{lemma}
\begin{proof}
To prove the lemma, it is sufficient to show that for each $d=d_1\otimes d_2'\otimes \widetilde{d}$ with $d\in \widehat{B_{n}}$ and $t(d)=l> k,$ and $c=c_1\otimes c_2'\otimes \widetilde{c}$ with $c\in \widehat{B_{n}}$ and $t(c)=k,$ the following property holds: $$(d_1\otimes d_2'\otimes \widetilde{d})(c_1\otimes c_2'\otimes \widetilde{c})\equiv b\otimes c_2'\otimes a\widetilde{c}~~~~(\mathrm{mod}~J_{k-2})$$ for some $b\in V(n,k);$ and the element $a$ is an element in $A_{k}$ which is independent of $\widetilde{c}.$ This property follows again by considering the composition of tangles as in the proof of Lemma 4.2. In fact, the concatenation $d_1\cdot \widetilde{d}\cdot d_{2}'\cdot c_1$ is an $(n, m)$-tangle with $m\leq k.$ If $m< k,$ then the composition of $d$ and $c$ is in $J_{k-2},$ so we are done. Suppose that $m=k,$ then this concatenation $d_1\cdot \widetilde{d}\cdot d_{2}'\cdot c_1$ is of the form $b\cdot a+a',$ where $a'$ is a linear combination of $(n, j)$-tangles with $j< k.$ If we concatenate this further with $\widetilde{c}\cdot c_{2}',$ then we get the desired statement.
\end{proof}

It is well-known how to define an involution $i$ on $\widehat{B}_{n, R}$: $i$ is given by reflection of a diagram through its horizontal axis. Note also that there is an involution $\sigma$ on $A_{n}$ defined by $\sigma(w)=w^{-1}$ for any $w\in W,$ where $W$ is the extended affine Weyl group of type $A_{n-1}.$ The following lemma easily follows from the fact that the extended affine Hecke algebra of type $A$ is affine cellular as is shown in [KX3] and from [KX3, Lemma 2.4].
\begin{lemma}
The group algebra $A_{n}$ over $\mathbb{Z}$ is an affine cellular $\mathbb{Z}$-algebra with respect to the involution $\sigma.$
\end{lemma}

The following lemma describes the effect of $i$ on an basis element of $\widehat{B}_{n, R}$ and $\sigma$ on the bilinear form $\varphi_{k},$ whose proof can be seen from the definition of $i$ and the tangle concatenations.
\begin{lemma}
$(1)$~ If $d=d_1\otimes d_2'\otimes \widetilde{d},$ then $i(d)=d_2\otimes d_1'\otimes \sigma(\widetilde{d}).$

$(2)$~ The involution $\sigma$ on $A_{k}$ and the bilinear form $\varphi_{k}$ have the following property: $\sigma\varphi_{k}(c', d)=\varphi_{k}(d', c)$ for all $c, d\in D(n,k).$
\end{lemma}

Combined with Lemma 2.1, 4.2, 4.3, 4.4 and 4.5, we have proved the main result of this note.

\begin{theorem}
The affine Brauer algebra $\widehat{B}_{n, R}$ over a noetherian domain $R$ is an affine cellular algebra with respect to the involution $i.$
\end{theorem}

Let now $A$ be an affine cellular algebra with a cell chain $0=J_{0}\subset J_{1}\subset J_{2}\subset\cdots\subset J_{n}=A$, such that each subquotient $J_{i}/J_{i-1}$ is an affine cell ideal of $A/J_{i-1}$. Then $J_{i}/J_{i-1}$ is isomorphic to $\mathbb{A}(V_{i},B_{i},\varphi_{i})$ for some free $k$-module $V_{i}$ of finite rank, an affine $k$-algebra $B_{i}$ and a $k$-bilinear form $\varphi_{i} :V_{i}\otimes_{k} V_{i}\rightarrow B_{i}.$ Let $(\varphi_{st}^{i})$ be the matrix representing the bilinear form $\varphi_{i}$ with respect to some choices of basis of $V_{i}.$ Then Koenig and Xi obtain a parameterisation of simple modules over an affine cellular algebra by establishing a bijection between isomorphism classes of simple $A$-modules and the set $$\{(j, m)| 1\leq j\leq n, m\in \mathrm{MaxSpec(B_{j})~ such~ that~ some} ~\varphi_{st}^{i}\notin \mathrm{m}\},$$
where $\mathrm{MaxSpec(B_{j})}$ denotes the maximal ideal spectrum of $B_{j}.$



Mathematical Sciences Center, Tsinghua University, Jin Chun Yuan West Building

Beijing, 100084, P. R. China.

E-mail address: cwdeng@amss.ac.cn

\end{document}